\documentclass[11pt,reqno]{amsart}

\usepackage[utf8]{inputenc}
\usepackage[T1]{fontenc}
\usepackage{lmodern}
\usepackage{enumitem}
\usepackage{fullpage}
\usepackage{hyperref}
\usepackage{stmaryrd, amsthm, mathtools, amssymb, dsfont, mathrsfs, bm,amsmath}
\usepackage{booktabs,diagbox}
\usepackage[noabbrev,capitalise]{cleveref}
\usepackage[ruled,nosemicolon,vlined]{algorithm2e}
\usepackage{tikz}
\usetikzlibrary{decorations.pathreplacing,positioning}
\usepackage{todonotes}
\usepackage{caption}
\usepackage{subcaption}
\usepackage{todonotes}
\usepackage{marginnote}
\let\oldmarginpar\marginpar
\renewcommand\marginpar[1]{\-\oldmarginpar[\raggedleft\footnotesize #1]%
{\raggedright\footnotesize #1}} 
\newtheorem{thm}{Theorem}[section]
\newtheorem{prop}[thm]{Proposition}
\newtheorem{cor}[thm]{Corollary}
\newtheorem{conj}[thm]{Conjecture}
\newtheorem{lemma}[thm]{Lemma}
\theoremstyle{definition}

\crefname{defi}{Definition}{Definitions}
\theoremstyle{remark}
\newtheorem{rk}{Remark}
\newcounter{cas}

\newtheoremstyle{assert}
  {.5\baselineskip±.2\baselineskip}   
  {.5\baselineskip±.2\baselineskip}   
  {\itshape}  
  {0pt}       
  {\bfseries} 
  {.}         
  {5pt plus 1pt minus 1pt} 
  {(\thmnumber{#2})}          
\theoremstyle{assert}

\DeclareMathOperator{\et}{and}

\newcommand{\restrict}[2]{{#1}_{\left|#2\right.}}

\newcommand{\eps}{\varepsilon}

\newcommand{\bsig}{{\bm{\sigma}}}

\newcommand{\bX}{\mathbf{X}}
\newcommand{\bx}{\mathbf{x}}

\newcommand{\bS}{\mathbf{S}}

\newcommand{\cH}{\mathcal{H}}

\newcommand{\cU}{\mathcal{U}}
\newcommand{\C}{\mathscr{C}}

\newcommand{\brc}[1]{\left\{#1 \right\}}
\newcommand{\pth}[1]{\left(#1 \right)}
\newcommand{\floor}[1]{\left\lfloor #1 \right\rfloor}
\newcommand{\ceil}[1]{\left\lceil #1 \right\rceil}
\newcommand{\pr}[1]{\mathbb{P}\left[ #1 \right]}
\newcommand{\esp}[1]{\mathbb{E}\left[ #1 \right]}
\newcommand{\midbar}{\;\middle|\;}

\newcommand{\sst}[2]{\left\{#1\,:\,#2\right\}}

\newcommand{\odd}{\chi_o}
\newcommand{\pcf}{\chi_{\rm pcf}}

\tikzstyle{vertex} = [draw,fill,shape=circle,node distance=80pt]
\tikzstyle{wertex} = [draw=black,fill=white,shape=circle,node distance=80pt]
\tikzstyle{gertex} = [draw=black,fill=black!25,shape=circle,node distance=80pt]
\tikzstyle{edge} = [fill,opacity=.5,fill opacity=.5,line cap=round, line join=round, line width=50pt]
\pgfdeclarelayer{background}
\pgfsetlayers{background,main}

\makeatletter
\newcommand{\newpar}[1]{%
    \par
    \addvspace{\medskipamount}
    \noindent\textit{#1\@addpunct{.}}\enspace\ignorespaces
    }
\makeatother

\title{New bounds for odd colourings of graphs}

\author{Tianjiao Dai}
\address{\'Equipe GALaC, LISN (Université Paris-Saclay),
Gif sur Yvette, France.}
\email{dai@lisn.fr}

\author{Qiancheng Ouyang}
\address{\'Equipe GALaC, LISN (Université Paris-Saclay),
Gif sur Yvette, France.}
\email{qiancheng.ouyang@lisn.fr}

\author{François Pirot}
\address{\'Equipe GALaC, LISN (Université Paris-Saclay),
Gif sur Yvette, France.}
\email{francois.pirot@lisn.fr}

\begin{document}

\begin{abstract}
Given a graph $G$, a vertex-colouring $\sigma$ of $G$, and a subset $X\subseteq V(G)$, a colour $x \in \sigma(X)$ is said to be \emph{odd} for $X$ in $\sigma$ if it has an odd number of occurrences in $X$. We say that $\sigma$ is an \emph{odd colouring} of $G$ if it is proper and every (open) neighbourhood has an odd colour in $\sigma$.
The odd chromatic number of a graph $G$, denoted by $\odd(G)$, is the minimum $k\in\mathbb{N}$ such that an odd colouring $\sigma \colon V(G)\to [k]$ exists. 
%
In a recent paper, Caro, Petru\v sevski and \v Skrekovski conjectured that every connected graph of maximum degree $\Delta\ge 3$ has odd-chromatic number at most $\Delta+1$. We prove that this conjecture holds asymptotically:
for every connected graph $G$ with maximum degree $\Delta$, $\odd(G)\le\Delta+O(\ln\Delta)$ as $\Delta \to \infty$. We also prove that $\odd(G)\le\lfloor3\Delta/2\rfloor+2$ for every $\Delta$. 
If moreover the minimum degree $\delta$ of $G$ is sufficiently large, we have $\odd(G) \le \chi(G) + O(\Delta \ln \Delta/\delta)$ and $\odd(G) = O(\chi(G)\ln \Delta)$.
Finally, given an integer $h\ge 1$, we study the generalisation of these results to $h$-odd colourings, where every vertex $v$ must have at least $\min \{\deg(v),h\}$ odd colours in its neighbourhood.
Many of our results are tight up to some multiplicative constant.
\end{abstract}

\maketitle

\vspace{-8pt}
\section{Introduction}



\noindent
All graphs considered here are finite, undirected, and simple. Let $[k]$ denote the set of first $k$ positive integers. For a graph $G$, let $V(G)$ and $E(G)$ denote its vertex and edge sets, respectively. A hypergraph $\cH=(V(\cH), E(\cH))$ is a generalisation  of a graph; its (hyper-)edges are subsets of $V(\cH)$ of arbitrary positive size. A vertex \emph{$k$-colouring} of a graph or a hypergraph $G$ is an assignment $\sigma \colon V(G)\to [k]$, whose images are referred to as \emph{colours}. Motivated by a frequency assignment problem in cellular networks, Even, Lotker, Ron, and Smorodinsky~\cite{even2003conflict} introduced the notion of conflict-free colourings of hypergraphs. A colouring $\sigma$ of a hypergraph $\cH$ is \emph{conflict-free} if for every edge $e\in E(\cH)$ there exists a colour appearing exactly once in $e$. Pach and Tardos \cite{pach2009conflict} studied
this notion and proved that every hypergraph with fewer than $\binom{s}{2}$ edges (for some integer $s$) has a
conflict-free colouring with fewer than $s$ colours. Kostochka, Kumbhat, and Luczak \cite{kostochka2012conflict}
further studied conflict-free colouring for uniform hypergraphs.

\subsection{A motivation coming from proper conflict-free colourings}
A colouring $\sigma$ of a graph $G$ is \emph{proper} if $\sigma(u)\neq \sigma(v)$ for all $uv\in E(G)$. The \emph{chromatic number} of a graph $G$, denoted by $\chi(G)$, is the minimum integer $k$ such that there is a $k$-colouring of $G$. 
One can observe that a conflict-free colouring of a hypergraph $\cH$ is in particular a proper colouring of the graph formed by the hyper-edges of size $2$ in $\cH$. Given a graph $G$ and a hypergraph $\cH$ on the same vertex set $V$, we say that $(G,\cH)$ is a \emph{graph-hypergraph pair}. A \emph{proper conflict-free $k$-colouring of $(G,\cH)$} (\emph{pcf $k$-colouring} for short) is a mapping $\sigma\colon V \to [k]$ that is both a proper colouring of $G$ and a conflict-free colouring of $\cH$. By the previous observation, this can be seen as a conflict-free colouring of the hypergraph $(V, E(\cH)\cup E(G))$.
There are many classical constraints for a proper colouring of a given graph $G$ that are in particular satisfied by a pcf-colouring of $(G,\cH)$, for some hypergraph $\cH$ carefully chosen. For instance, if $E(\cH)$ contains a maximum independent set of every even cycle in $G$, then a pcf-colouring of $(G,\cH)$ is in particular an acyclic colouring of $G$. 
As another example, if $E(\cH)$ contains all $(\beta+1)$-subsets of every neighbourhood in $G$, then a pcf-colouring of $(G,\cH)$ is in particular a $\beta$-frugal colouring of $G$.

There has been a specific focus on the special case of pcf-colourings $\sigma$ of $(G,\cH)$  when $\cH$ is the neighbourhood-hypergraph of $G$, i.e. $E(\cH) = \{N(v) : v\in V(G), \deg(v)>0\}$.
In that case, we say that $\sigma$ is a pcf-colouring of $G$ (so when we omit $\cH$, it is implicitly the neighbourhood-hypergraph of $G$). In other words, a \emph{pcf-colouring} of $G$ is a proper colouring of $G$ such that for every non-isolated vertex $v$, there is a colour that appears exactly once among the neighbours of $v$. We let $\pcf(G)$ be the smallest integer $k$ such that a pcf $k$-colouring of $G$ exists. This notion is the combination of proper colouring and the pointed conflict-free chromatic parameter introduced by Cheilaris \cite{Chei09}.

The notion of pcf-colourings of graphs was formally introduced by Fabrici, Lu\v{z}ar, Rindo\v{s}ov\'{a}, and Sot\'{a}k \cite{Fab23}, where they investigated the pcf-colourings of planar and outerplanar
graphs, among many other related variants of a proper conflict-free colouring. They proved that $\pcf(G)\le 8$ for all planar graphs and $\pcf(G)\le 5$ for all outerplanar graphs. Further studies in pcf-colourings of sparse graphs can be found in \cite{CPS23,CCKP22+pcf,Fab23,Hic22+,Liu22+}.

Given a (hyper-)graph $\cH$, 
the \emph{degree} of a vertex $v$, denoted by $\deg_\cH(v)$, is the number of edges of $\cH$ containing $v$.
We denote by $\delta(\cH)$ and $\Delta(\cH)$ the minimum and maximum degrees of $\cH$, respectively. 
We denote $\epsilon(\cH)$ the minimum size of an edge in $\cH$.
Given a graph $G$, we let $\delta^*(G)$ denote the \emph{degeneracy} of $G$, that is $\delta^*(G) = \max_{H\subseteq G} \delta(H)$.
Caro, Petru\v{s}evski, and \v{S}krekovski \cite{CPS23} proposed the following conjecture about pcf-colourings.

\begin{conj}[Caro, Petru\v{s}evski, \v{S}krekovski {\cite[Conjecture 6.4]{CPS23}}]\label{conj:pcf}
    If $G$ is a connected graph of maximum degree $\Delta\ge 3$, then $\pcf(G)\le \Delta+1$.
\end{conj}

As a first step toward their conjecture, Caro, Petru\v{s}evski, and \v{S}krekovski \cite{CPS23} proved that for such a graph $G$, $\pcf(G)\le\floor{2.5\Delta}$. Recently, it has been observed by Cranston and Liu \cite{CrLi22+} that $\pcf(G) \le \Delta(G)+\delta^*(G)+1$ (they actually more generally proved that there always exists a pcf $(\Delta(\cH)+\delta^*(G)+1)$-colouring of any given pair $(G,\cH)$).
They further reduced the gap to Conjecture~\ref{conj:pcf} by proving that $\pcf(G)\le \ceil{1.6550826\Delta+ \sqrt{\Delta}}$, given that $\Delta$ is large enough. 


\subsection{Odd colourings}
In \cite{Chei13}, Cheilaris, Keszegh, and P\'{a}lv\"{o}lgyi introduced a weakening of conflict-free colourings. An \emph{odd colouring} $\sigma$ of a hypergraph $\cH$ satisfies the constraint that in every edge $e\in E(\cH)$ there is a colour $x$ with an odd number of occurrences in $\sigma$; we say that $x$ is an \emph{odd colour} of $e$ in $\sigma$. It is straightforward that a conflict-free colouring of $\cH$ is in particular an odd colouring of $\cH$.
Petru\v{s}evski and \v{S}krekovski \cite{PeSk22} later considered that notion applied to the neighbourhood-hypergraph of a graph $G$. An \emph{odd colouring} of $G$ is a proper colouring of $G$ with the additional constraint that each non-isolated vertex has a colour appearing an odd number of times in its neighbourhood. The \emph{odd chromatic number} of $G$, denoted by $\odd(G)$, is the minimum integer $k$ such that an odd $k$-colouring of $G$ exists. Since odd colourings are a weakening of pcf-colourings, it always holds that $\odd(G) \le \pcf(G)$. In the last couple of years, there has been some interest in determining the extremal value of $\odd$ in various classes of graphs.

In \cite{PeSk22}, Petru\v{s}evski and \v{S}krekovski showed that $\odd(G)\le 9$ for every planar graph $G$ with a proof that relies on the discharging method (note that this is weaker than the result about pcf-colourings of planar graphs from \cite{CCKP22+pcf} stated earlier, which was proved a couple of months later). Furthermore, they conjectured that this bound may be reduced to $5$.
If true, this would be tight, since $\odd(C_5)=5$. 
Recently there has been considerable attention in odd colourings of planar graphs \cite{CPS22,CCKP23,Cra22+,CLS23,PePo23}.



Caro, Petru\v{s}evski, and \v{S}krekovski \cite{CPS22} also studied various properties of the odd chromatic number of general graphs; in particular, they proved the following facts: every graph of maximum degree three has an odd 4-colouring; every graph, except for $C_5$, of maximum degree $\Delta$ has an odd $2\Delta$-colouring. Moreover, they proposed a conjecture for general graphs, which is a weaker form of Conjecture~\ref{conj:pcf}.

\begin{conj}[Caro, Petru\v{s}evski, \v{S}krekovski {\cite[Conjecture 5.5]{CPS22}}]\label{conj:graph}
If $G$ is a connected graph of maximum degree $\Delta\ge3$, then $\odd(G)\le\Delta+1.$
\end{conj}

Our main result states that Conjecture \ref{conj:graph} holds asymptotically as $\Delta\to \infty$.

\begin{thm}\label{thm:main}
For every graph $G$ of maximum degree $\Delta$,
\[\odd(G)\le\Delta+ \ceil{4(\ln\Delta+\ln\ln \Delta +3)}.\]
\end{thm}

For small values of $\Delta$, we provide another bound on $\odd(G)$ that is derived from a relatively simple (deterministic) colouring procedure.

\begin{thm}\label{thm:smalldegree}
    For every graph $G$ of maximum degree $\Delta$,
    \[\odd(G)\le\floor{\frac{3\Delta}{2}}+2.\]
\end{thm}

\subsection{General hypergraphs}
The proofs of our results stated in Theorem~\ref{thm:main} and Theorem~\ref{thm:smalldegree} rely highly on the structure of neighbourhood hypergraphs. In a more general setting, we could wonder how the odd colouring problem behaves on any graph-hypergraph pair $(G,\cH)$. We were able to extend Theorem~\ref{thm:main} to that more general setting, at the cost of requiring a lower bound on the minimum hyper-edge size $\epsilon(\cH)$ in $\cH$. 

\begin{thm}
\label{thm:hypergraph-delta}
    There exists a universal constant $C$ such that, for every graph-hypergraph pair $(G,\cH)$, if $\epsilon(\cH) \ge C \log \Delta(\cH)$ then there is an odd $k$-colouring of $(G,\cH)$, where
    \[ k \le \Delta(G) + C\log \Delta(\cH). \]
\end{thm}

\noindent
With that extra condition on $\epsilon(\cH)$, we can actually derive an upper bound on $\odd(G,\cH)$ that mainly depends on $\chi(G)$ rather than $\Delta(G)$. Moreover, we show with a construction that this bound is tight up to the precise value of the constant $C$.

\begin{thm}
\label{thm:hypergraph-chi-mult}
    There exists a universal constant $C$ such that, for every graph-hypergraph pair $(G,\cH)$, if $\epsilon(\cH) \ge C \log \Delta(\cH)$ then there is an odd $k$-colouring of $(G,\cH)$, where
    \[ k \le \chi(G)\cdot C \log \Delta(\cH). \]
\end{thm}

\noindent
When $\epsilon(\cH)$ gets closer to $\Delta(G)$, we show that the difference $\odd(G,\cH)-\chi(G)$ gets relatively small.

\begin{thm}
\label{thm:hypergraph-chi}
    There exists a universal constant $C$ such that, for every graph-hypergraph pair $(G,\cH)$, there is an odd $k$-colouring of $(G,\cH)$, where
    \[ k \le \chi(G) + C\, \frac{\Delta(G)\log \Delta(\cH)}{\epsilon(\cH)}. \]
\end{thm}

A direct consequence of Theorem~\ref{thm:hypergraph-chi} is that for quasi-regular graphs $G$ (that is, the ratio $\Delta(G)/\delta(G)$ is bounded by a uniform constant), the difference $\chi_o(G) - \chi(G)$ is small, namely $O(\log \Delta(G))$. This contrasts with the general case where that difference can be much larger: if $G$ is the $1$-subdivision of the complete graph on $\Delta+1 \ge 5$ vertices, then $\chi(G)=2$ while $\odd(G) = \Delta+1$. 






\subsection{Organisation of the paper}

The paper is organised as follows.
In Section~\ref{sec:proba}, we provide probabilistic tools that we will rely on in our proofs.
In Section~\ref{sec:pcf}, we prove Theorem~\ref{thm:smalldegree}.
In Section~\ref{sec:main}, we prove Theorem~\ref{thm:main}. In Section~\ref{sec:hypergraph}, we analyse odd colourings for general hypergraphs with additional constraints on the edge cardinalities.
Finally, in Section~\ref{sec:general}, we extend our results from Section~\ref{sec:hypergraph} to $h$-odd colourings, a generalisation of odd colourings that consists in requiring that every hyper-edge $e$ contains at least $\min\{|e|,h\}$ odd colours, for a given integer $h\ge 1$.

\section{Probabilistic tools.}
\label{sec:proba}
\noindent The first probabilistic result that we need is the following lopsided version of the Symmetric Lov\'{a}sz Local Lemma (LLL for short) (see e.g. \cite{Ber19}).
\begin{lemma}[Lopsided Lov\'{a}sz Local Lemma]\label{lemma:lll}
Let $\mathcal{B} = \{B_1, \ldots, B_n\}$ be a finite set of random (bad) events, and let $d$ be a fixed integer. Suppose that, for every $i\in [n]$, there is a set $\Gamma(i)\subseteq [n]$ of size at most $d$ such that, for every $Z\subseteq [n]\setminus\Gamma(i)$,
\[\pr{B_i \; \middle| \; \bigcap_{j\in Z}\overline{B_j}}\le p.\]
If $epd\le 1$, then $\pr{\bigcap_{i\in [n]} \limits \overline{B_i}}>0$.
\end{lemma}

Many random variables we analyse in this paper are highly concentrated around their mean. This is a consequence of Chernoff's bounds as stated hereafter.

\begin{lemma}[Chernoff's bounds]
    \label{lem:chernoff}
    Let $\bX_1, \ldots, \bX_n$ be i.i.d. $(0,1)$-valued random variables, and let $\bS_n \coloneqq \sum_{i=1}^n \bX_n$. Let us write $\mu\coloneqq \esp{\bS_n}$. Then
    \begin{enumerate}[label={\rm(\roman*)}]
        \item for every $0 \le \eta < \mu$, 
        \[\pr{\bS_n \le \mu - \eta} \le e^{-\frac{\eta^2}{2\mu}} \mbox{, and}\]
        \item for every $\eta > 0$, 
        \[ \pr{\bS_n \ge \mu + \eta} \le e^{-\frac{\eta^2}{2(\mu+\eta)}}.\]
        
    \end{enumerate}
\end{lemma}

Finally, we will need to analyse a specific Markovian process as described in the following lemma. 
First we observe that, using the well-known Stirling bounds
\[\sqrt{2\pi n}\pth{\frac{n}{e}}^n e^{\frac{1}{12n+1}} \le n! \le \sqrt{2\pi n}\pth{\frac{n}{e}}^n e^{\frac{1}{12n}}\] 
for every integer $n\ge 1$, it is straightforward to derive that 
\begin{equation}
\label{eq:factorial}
\frac{n!}{\pth{\frac{n}{2}}!} \le \sqrt{2}\pth{\frac{2n}{e}}^{n/2}
\end{equation}
for every even integer $n \ge 2$. We will also use the classical upper bound $\binom{n}{k}\le \pth{\frac{ne}{k}}^k$.
\begin{lemma}\label{lemma:be}
Let $\bX_1,\bX_2,\ldots$ be a sequence of binary random variables with values in $\{-1,1\}$, let $S_0\in\mathbb{N}$ be a non-negative integer, and let $\bS_i \coloneqq S_0+\sum_{j=1}^i \bX_j$ for every integer $i\ge 0$. If there exists a real number $\tau>0$ such that, for every $i\ge 0$, it holds that $\pr{\bX_i=-1 \midbar \bS_{i-1}=s}\le s/ \tau$, then 
\[\pr{\bS_n\le k}\le \sqrt{2} \, \binom{n}{k}\pth{\frac{2n-2k}{e\tau}}^{\frac{n-k}{2}},\]
for every integers $0\le k \le n$.
\end{lemma}

\begin{proof}[Proof of \cref{lemma:be}]
Given a possible outcome $(X_1, \ldots X_n)$ of $(\bX_1,\ldots,\bX_n)$ that yields that $\bS_n\le k$, let $I \coloneqq \{ i : X_i=-1 \}$. Noting that $\bS_n = S_0 + n - 2|I|$ conditioned on $\bX_1=X_1, \ldots, \bX_n=X_n$, we infer that $|I| \ge \frac{S_0+n-k}{2} \ge \frac{n-k}{2}$. 
Moreover, if $i_j$ is the $j$-th last element in $I$, then it deterministically holds that $\bS_{i_j-1} \le k+j$, and hence 
\[\pr{\bX_{i_j} = -1\midbar \bX_1=X_1, \ldots, \bX_{i_j-1}=X_{i_j-1}} \le  \pr{\bX_{i_j} = -1 \midbar \bS_{i_j-1} \le k+j} \le \frac{k+j}{\tau}.\]
We therefore have the following (crude) upper-bound:
\begin{align*}
    \pr{\bX_1=X_1, \ldots, \bX_n=X_n}
    &\le \prod_{i\in I} \pr{\bX_i = -1 \midbar \bX_1=X_1, \ldots, \bX_{i-1}=X_{i-1}}\\
    &\le  \prod_{j=1}^{\frac{n-k}{2}}\frac{k+j}{\tau} = \frac{\pth{\frac{n+k}{2}}!}{k!\, \tau^\frac{n-k}{2}}.
\end{align*}  
Since there are at most $\binom{n}{\frac{n-k}{2}}$ possibles choices for the $\frac{n-k}{2}$ last elements of $I$, we have
\begin{align*}
\pr{\bS_n\le k}&\le \binom{n}{\frac{n-k}{2}}\cdot\frac{\pth{\frac{n+k}{2}}!}{k!\, \tau^\frac{n-k}{2}}
= \frac{n!}{k!\pth{\frac{n-k}{2}}!\,\tau^\frac{n-k}{2}} = \binom{n}{k} \frac{(n-k)!}{\pth{\frac{n-k}{2}}!\, \tau^{\frac{n-k}{2}}} \\
&\le \sqrt{2}\binom{n}{k} \pth{\frac{2n-2k}{e\tau}}^{\frac{n-k}{2}} & \mbox{by \eqref{eq:factorial}}\\
&\le \sqrt{2} \pth{\frac{ne}{k}}^k \pth{\frac{2n-2k}{e\tau}}^{\frac{n-k}{2}}.
\end{align*}
\end{proof}

Let $G$ be a graph, and $\C(G) \subseteq [k]^{V(G)}$ a prescribed set of $k$-colourings of $G$. Given a colouring $\sigma \in \C(G)$ and a vertex $v\in V(G)$, we let $L_\sigma(v)$ be the set of colours $x$ such that, if we redefine $\sigma(v) \gets x$, we still have $\sigma \in \C(G)$. For instance, if $\C(G)$ is the set of proper $k$-colourings of $G$, then $L_\sigma(v)=[k]\setminus \sigma(N(v))$.

\begin{lemma}
\label{lem:badevent}
Let $G$ be a graph, and $\C(G)$ a set of colourings of $G$. 
Let $\bsig$ be drawn uniformly at random from $\C(G)$.
For a given subset of vertices $X\subseteq V(G)$, let $B_\bsig$ be the bad event that $X$ has no odd colour in $\bsig$, and let $M\subseteq X$ be a subset of size $m\le |X|$.
If there exists an integer $\tau$ such that we deterministically have $|L_\bsig(v)|\ge \tau$ for every vertex $v\in M$, then for every possible realisation $\sigma_0$ of $\restrict{\bsig}{V(G)\setminus M}$, we have
\[ \pr{B_\bsig  ~\middle|~ \restrict{\bsig}{V(G)\setminus M}=\sigma_0} \le \sqrt{2} \pth{\frac{2m}{e\tau}}^{m/2}.\]
\end{lemma}

\begin{proof}
    Let $M=\{u_1, \ldots, u_m\}$ be a fixed subset of $X$. Let $\sigma_0$ be a possible realisation of $\restrict{\bsig}{V(G)\setminus M}$.
    Let $\bsig_0$ be drawn uniformly at random from the extensions of $\sigma_0$ to $\C(G)$.
    For every $1\le i \le m$, we let $\bsig_i \in \C(G)$ be obtained from $\bsig_{i-1}$ by resampling the colour of $u_i$ uniformly at random from $L_{\bsig_{i-1}}(u_i)$.
    For every $i\le m$, let $\bS_i$ be the number of odd colours of $X$ in $\bsig_i$.
    For every $i\ge 1$, we have $\bS_i = \bS_{i-1}-1$ if $\bsig_i(u_i)$ is one of the $\bS_{i-1}$ odd colours of $X$ in $\bsig_{i-1}$; since there are at least $\tau$ choices for $\bsig_i(u_i)$ this happens with probability at most $k/\tau$ if $\bS_{i-1}=k$. Otherwise, we have $\bS_i = \bS_{i-1}+1$. So the sequence $(\bS_i)_{i\le m}$ satisfies the hypotheses of Lemma~\ref{lemma:be}, hence by setting $k\coloneqq 0$ we have 
    \[ \pr{B_{\bsig_m}}=\pr{\bS_m = 0} \le \sqrt{2}\pth{\frac{2m}{e\tau}}^{m/2}.\]

    Since we resample the colours uniformly at random, the random colourings $(\bsig_i)_{i\le m}$ are identically distributed. Therefore, if $\bsig$ is drawn uniformly at random from $\C(G)$, we have 
    \[\pr{B_\bsig ~\middle|~ \restrict{\bsig}{V(G)\setminus M} = \sigma_0} = \pr{B_{\bsig_0}} = \pr{B_{\bsig_m}},\]
    and the conclusion follows.
\end{proof}

\section{A greedy bound}
\label{sec:pcf}
 Given a proper $k$-colourings $\sigma\colon V(G) \to [k]$, and a vertex $u\in V(G)$, we denote $\cU_\sigma(u)$ the set of odd colours of  $N_G(u)$ in $\sigma$. So $\sigma$ is an odd $k$-colouring if $|\cU_\sigma(u)|\ge 1$ for every vertex $u\in V(G)$.
If $\cU_\sigma(u)= \{x\}$, we say that $u$ is \emph{$\sigma$-critical}, and that $x$ is its \emph{witness colour}; we denote it $w_\sigma(u)\coloneqq x$.

\begin{proof}[Proof of Theorem~\ref{thm:smalldegree}]
    Let $v_1, \ldots, v_n$ be an arbitrary ordering of the vertices of $G$. We let $H_i \coloneqq G[\{v_1, \ldots, v_i\}]$ for every $i\in [n]$.
    Let $k \coloneqq \floor{\frac{3\Delta}{2}}+2$, and let $\C(H)$ denote the set of odd $k$-colourings of each induced subgraph $H \subseteq G$.
    We construct an odd $k$-colouring of $G$ greedily by constructing a sequence of partial colourings $(\sigma_i)_{i\in [n]}$ that satisfies the following induction hypothesis.
    \begin{equation}
        \label{eq:invariant}
        \tag{IH}
        \sigma_i \in \C(H_i) \mbox{ and } |\cU_{\sigma_i}(u)|\ge 1 \mbox{ for every vertex $u\in N_G(V(H_i))$.}
    \end{equation}
    For the base case, we may begin with the empty colouring $\sigma_0$.
    Let us now assume that we have constructed $\sigma_i$ that satisfies \eqref{eq:invariant}.
    In order to maintain \eqref{eq:invariant}, we forbid that $\sigma_{i+1}(v_{i+1})$ is one of $\sst{\sigma_i(u)}{u\in N_{H_i}(v_{i+1})} \cup \sst{w_{\sigma_i}(u)}{u \in N_G(v_{i+1}) \mbox{ and $u$ is $\sigma_i$-critical}}$.
    If at most $k-1$ colours are forbidden for $v_{i+1}$, then there remains at least one possible choice for $\sigma_{i+1}(v_{i+1})$, and letting $\sigma_{i+1}(u) = \sigma_i(u)$ for every $u\in V(H_i)$ we have that $\sigma_{i+1}$ satisfies \eqref{eq:invariant}.
    
    We may now assume that all $k$ colours are forbidden for $v_{i+1}$. Let $X\subseteq N_G(v_{i+1})$ be the set of neighbours of $v_{i+1}$ that forbid exactly one colour for $v_{i+1}$, and $Y\subseteq N_G(v_{i+1}) \setminus X$ be the set of neighbours of $v_{i+1}$ that forbid exactly two colours for $v_{i+1}$ (so $Y \subseteq V(H_i)$ and every vertex $y\in Y$ is $\sigma_i$-critical). We claim that there is a vertex $y \in Y$ such that $\sigma_i(y)$ is forbidden only by $y$ for $v_{i+1}$. Indeed, otherwise the number of forbidden colours for $v_{i+1}$ would be at most $|X| + \frac{3}{2}|Y| \le \frac{3}{2}\Delta < k$, a contradiction.
    We also claim that $\sigma_i(y)$ is not a witness colour of $v_{i+1}$ in $\sigma_i$. Indeed, otherwise every colour of $N(v_{i+1})\setminus \{y\}$ appears at least twice, hence $v_{i+1}$ has at most $\floor{\frac{\Delta-1}{2}}+1$ adjacent colours in $\sigma_i$. Since there are at most $\Delta$ witness colours in $N_G(v_{i+1})$, there are at most $\floor{\frac{3\Delta+1}{2}} \le k-1$ forbidden colours for $v_{i+1}$, a contradiction.
    Let us set $\sigma_{i+1}(v_{i+1}) \coloneqq \sigma_i(y)$, and $\sigma_{i+1}(u) \coloneqq \sigma_i(u)$ for every vertex $u\in V(H_i)\setminus \{y\}$.
    
    There remains to define $\sigma_{i+1}(y)$. Since $|\cU_{\sigma_i}(y)|=1$, it means that every colour in $N_G(y)\cap V(H_{i+1})$ appears at least twice in $\sigma_{i+1}$ except $w_{\sigma_i}(y)$ and $\sigma_{i+1}(v_{i+1})=\sigma_i(y)$. So $y$ has at most $\lfloor \frac{\Delta-2}{2}\rfloor+2 = \lfloor\frac{\Delta}{2}\rfloor+1$ adjacent colours in $\sigma_{i+1}$. Since $N_G(y)$ contains at most $\Delta$ witness colours in $\sigma_{i+1}$, there are at most $\lfloor\frac{3\Delta}{2}\rfloor+1 = k-1$ forbidden colours for $y$ in $\sigma_{i+1}$, and so there remains at least one possible choice for $\sigma_{i+1}(y)$.
    This ends the proof of the induction.

    We conclude that $\sigma_n$ is an odd $k$-colouring of $G$, which proves that $\odd(G)\le k$, as desired.
\end{proof}

\section{An asymptotic version of the Odd Colouring Conjecture}
\label{sec:main}

In this section we prove Theorem~\ref{thm:main}.

\subsection{Set-up} 
Let $G$ be a connected graph of maximum degree $\Delta$.
Given a vertex-colouring $\sigma \colon V(H)\to [k]$ of some induced subgraph $H$ of $G$, for each $v\in V(H)$ a colour $x$ is said to be an \emph{odd colour} of $v$ if $x$ is an odd colour of $N_H(v)$ in $\sigma$. Let $w_\sigma(v)$ denote the unique odd colour of $v$ in $\sigma$ if such a colour exists; otherwise $w_\sigma(v)$ is undefined. When it exists, we say that $w_\sigma(v)$ is the \emph{witness colour} of $v$ in $\sigma$.

Let $k>\Delta$ be some integer. Let $V^-$ be the subset of $V(G)$ consisting of all vertices of degree less than $k/2$, and $V^+\coloneqq V(G)\setminus V^+$ be the set of vertices of degree at least $k/2$.
We denote $G^+ \coloneqq G[V^+]$.
For every $X \subseteq V^+$, we say that a proper partial $k$-colouring $\sigma\colon X\to [k]$ of $G^+$ is \emph{admissible} if every vertex $v\in V^-$ having $N_G(v) \subseteq X$ has an odd colour in $\sigma$.
Finally, we let $V^{++}$ be the set of vertices $v\in V^+$ having $N_G(v) \subseteq V^+$.

Since the bound on $\odd(G)$ that we want to prove is weaker than that of Corollary~\ref{thm:smalldegree} if $\Delta \le 65$, we may assume that $\Delta \ge 66$.

\subsection{Colouring vertices of large degree.} 
Let $\bsig\colon V^+\to [k]$ be a uniformly random admissible colouring of $G^+$.
For every $v\in V^{++}$, we let $B_\bsig(v)$ be the random event that $N_G(v)$ has no odd colour in $\bsig$. The goal of this subsection will be to show that, with non-zero probability, no event $B_\bsig(v)$ occurs.

\begin{lemma}\label{lemma:admissible}
If $\Delta \ge 49$ and $k \ge \Delta+ 4(\ln \Delta + \ln \ln \Delta + 3)$, then there exists an admissible colouring $\sigma \colon V^+\to [k]$ of $G^+$ such that every vertex $v\in V^{++}$ has an odd colour in $\sigma$.
\end{lemma}

\begin{proof}[Proof of \cref{lemma:admissible}]
Fix $k \coloneqq \Delta + \eta$, for some integer $\eta \ge 1$ whose precise value will be determined later in the proof, and let $\bsig$ be a uniformly random admissible $k$-colouring of $G^+$. Such a colouring exists, since each vertex $v$ has at most $\Delta$ constraints (at most $\deg_{V^+}(v)$ constraints because of the adjacent colours, and at most $\deg_{V^-}(v)$ constraints because of the adjacent witness colours). In particular, we have $|L_\bsig(v)|\ge \eta$ for every $v\in V^+$.
We want to show that, with non-zero probability, no bad event $B_\bsig(v)$ occurs for $v\in V^{++}$.

Let $m \le k/2$ be some integer whose explicit value will be determined later in the proof.
For every $v \in V^{++}$, we pick an arbitrary subset $M(v)\subseteq N(v)$ of size $m$. 
Then we let $\Gamma(v)\coloneqq N[M(v)]$.
For a vertex $u\in V^{++}$, the outcome of $B_\bsig(u)$ is entirely determined by the colours assigned to vertices in $N(u)$. So if we fix the realisation of $\bsig$ outside of $M(v)$, we in particular fix the outcomes of all events $B_\bsig(u)$ such that $M(v) \cap N[u] = \emptyset$. This holds for every $u\notin \Gamma(v)$.
We wish to apply Lemma~\ref{lemma:lll} to those bad events, with that definition of $\Gamma(v)$.
To that end, let $\Sigma_0$ be the set of possible realisations of $\restrict{\bsig}{V^+ \setminus M(v)}$ such that no event $B_\bsig(u)$ occurs for $u \notin \Gamma(v)$. For every $Z\subseteq V^{++} \setminus \Gamma(v)$, we have
\[\pr{B_\bsig(v) \midbar \bigcap_{u \in Z} \overline{B_\bsig(u)}} \le \sup_{\sigma_0 \in \Sigma_0} \pr{B_\bsig(v) \midbar \restrict{\bsig}{V^+\setminus M(v)}=\sigma_0} \le \sqrt{2} \pth{\frac{2m}{e\eta}}^{m/2},\]
by Lemma~\ref{lem:badevent} applied to the graph $G^+$ with $\C(G^+)$ being the set of admissible $k$-colourings of $G^+$.



Let us fix $\eta \coloneqq 2m$, so that this is at most $\sqrt{2}e^{-m/2}$. Since $v\in V^{++}$, we know that $\deg_{G^+}(v) \ge k/2 \ge \eta/2 \ge m$, so this lets us pick any value for $m$. To apply Lemma~\ref{lemma:lll}, we need an upper bound of $\frac{1}{em\Delta}$ for that probability, which holds precisely when $m \ge -2 W_{-1}\pth{-\frac{1}{2\sqrt{2}e\Delta}}$. We may therefore pick $m \coloneqq \ceil{-2 W_{-1}\pth{-\frac{1}{2\sqrt{2}e\Delta}}}$; a careful analysis of that value yields that $2m \le \ceil{4(\ln \Delta + \ln \ln \Delta + 3)}$ when $\Delta \ge 49$.
\end{proof}

\subsection{Colouring vertices of small degree.} By \cref{lemma:admissible}, $G^+$ has an admissible $k$-colouring $\sigma \colon V^+\to [k]$ with $k=\Delta+\ceil{4(\ln \Delta + \ln \ln \Delta + 3)}$, such that every $v\in V^{++}$ has an odd colour. Then we colour the vertices of $V^-$ greedily. Each time we colour $v\in V^-$, each neighbour $u$ of $v$ yields at most $2$ forbidden colours (its colour $\sigma(u)$, and its witness colour if it exists), so there are less than $k$ forbidden colours for $v$. This ensures that the greedy colouring terminates, and ends the proof of Theorem~\ref{thm:main}.

\section{Odd colourings of hypergraphs with constrained edge sizes}
\label{sec:hypergraph}
\noindent We recall that, given a graph-hypergraph pair $(G,\cH)$, an \emph{odd $k$-colouring of $(G,\cH)$} is a mapping $\sigma\colon V \to [k]$ that is both a proper colouring of $G$ and an odd colouring of $\cH$.
Given a set $S\subseteq V(\cH)$, we let $\cH[S] = (S, \{e \cap S : e\in E(\cH) \})$ be the sub-hypergraph of $\cH$ induced by $S$ (note that this definition allows the possibility that $\cH[S]$ contains $\emptyset$ as an edge).

\subsection{A bound in terms of the chromatic number for quasi-regular graphs}
\label{sec:qreg}
\noindent The proof of Theorem~\ref{thm:main} relies on the probabilistic method by analysing the behaviour of a uniformly random admissible colouring of a given graph $G$. It turns out that we have exploited the randomness of only a subset of the vertices of $G$: a subset of $m$ neighbours of each vertex $v\in V^{++}$, where $m = \Theta(\ln \Delta(G))$. If $G$ has a large minimum degree, we may restrict the random choices to a small subset of vertices that should suffice to have an odd colour in every neighbourhood, and colour the other vertices with an optimal proper colouring. 

\begin{thm}
    \label{thm:chi-bound}
    Let $(G,\cH)$ be a graph-hypergraph pair. Fix $\eta \coloneqq \ceil{4(\ln \Delta(\cH) + \ln \ln \Delta(\cH) + 3)}$, and assume that $\Delta(\cH) \ge 49$.
    For every subset of vertices $S\subseteq V(\cH)$, if $\epsilon(\cH[S]) \ge \eta/2$, then $(G,\cH)$ has an odd $k$-colouring, where 
    \[ k \le \chi(G\setminus S) + \Delta(G[S]) + \eta.\]
\end{thm}

\begin{proof}
    Let $k_0 \coloneqq \chi(G\setminus S)$ and let $\sigma_0$ be a proper $k_0$-colouring of $G\setminus S$.
    Let $k \coloneqq k_0 + \Delta(G[S]) + \eta$, and let $\bsig$ be a uniformly random proper $k$-colouring of $G$ that satisfies $\restrict{\bsig}{G\setminus S}=\sigma_0$. For every $e\in E(\cH)$, we let $B_{\bsig}(e)$ be the random event that $e$ has no colour appearing at odd times in $\bsig$. Let us show that, with non-zero probability, no event $B_{\bsig}(e)$ occurs. 
    Let $m\le \eta/2$ be an integer, and for every edge $e \in E(\cH)$ let $M(e) = \{u_1,u_2,\ldots, u_m\} \subseteq e\cap S$ be a subset of $m$ vertices in $e$. Let us recolour the vertices in $M(e)$ in turn with a uniformly random available colour. Each time we recolour $u_i$, the neighbours of $u_i$ in $S$ forbid at most $\deg_S(u_i) \le \Delta(G[S])$ colours, and the neighbours of $u_i$ not in $S$ forbid at most $k_0 = \chi(G\setminus S)$ colours (these colours are fixed by $\sigma_0$). So there are at least $\eta$ available colours for $u_i$. In particular, we have $|L_{\bsig}(v)|\geq\eta$ for each $v\in V(G)$. 

    We apply Lemma~\ref{lemma:lll} with $\Gamma(e) \coloneqq  \{ e' \in E(\cH) : e' \cap M(e) \neq \emptyset\}$ for every edge $e\in E(\cH)$, and obtain that, with non-zero probability, none of the events $B_{\bsig}(e)$ occurs. The size of $\Gamma(e)$ is at most $m\Delta(\cH)$. Let $\Sigma_1$ be the set of possible realisations of $\restrict{\bsig}{V(G) \setminus M(e)}$ such that no event $B_\bsig(e')$ occurs for $e' \notin \Gamma(e)$. For every $Z\subseteq E(\cH) \setminus \Gamma(e)$, we have
\[\pr{B_{\bsig}(e) \midbar \bigcap_{e' \in Z} \overline{B_{\bsig}(e')}} \le \sup_{\sigma_1 \in \Sigma_1}\pr{B_{\bsig}(e) \midbar \restrict{\bsig}{V(G)\setminus M(e)}=\sigma_1} \le \sqrt{2}\pth{\frac{2m}{e\eta}}^{m/2}= \sqrt{2}\,e^{-m/2},\]
 by Lemma~\ref{lem:badevent} applied to the graph $G$ with $\C(G)$ being the set of proper $k$-colourings of $G$. As in the proof of Lemma~\ref{lemma:admissible}, by fixing $m\coloneqq \ceil{-2 W_{-1}\pth{-\frac{1}{2\sqrt{2}e\Delta(\cH)}}} \le \eta/2$, the above probability is at most $\frac{1}{e\cdot m\Delta(\cH)}$. This proves the existence of a proper $k$-colouring $\sigma$ of $G$ such that every vertex has an odd colour in $\sigma$, as desired.


\end{proof}

By taking $S=V(G)$, Theorem~\ref{thm:chi-bound} has the following result as a corollary.

\begin{cor}
    Let $(G,\cH)$ be a graph-hypergraph pair, and fix $\eta \coloneqq \ceil{4(\ln \Delta(\cH) + \ln \ln \Delta(\cH) + 3)}$. If $\Delta(\cH) \ge 49$ and $\epsilon(\cH) \ge \eta/2$, then there exists an odd $(\Delta(G)+\eta)$-colouring of $(G,\cH)$.
\end{cor}

We next show how to find a set $S$ satisfying the hypothesis of Theorem~\ref{thm:chi-bound} such that $\Delta(G[S])$ is as small as possible. 

\begin{lemma}
\label{lem:degree-dominating-set}
Let $(G,\cH)$ be a graph-hypergraph pair.
Let $\Delta \coloneqq \Delta(G)+\Delta(\cH)$, and
let us assume that the minimum edge size in $\cH$ is $\epsilon(\cH) \ge 12\ln\Delta$.
Let $r \coloneqq \min \{ \epsilon(\cH), \Delta(G)\}$.
Then for every $m$ satisfying $4\ln\Delta\le m \le r/3$, if $\Delta$ is large enough, there is a subset $S\subseteq V(G)$ such that $\epsilon(\cH[S]) \ge m$, and  $\Delta(G[S]) \le \frac{\Delta(G)}{r}(m+\sqrt{60\, m\ln \Delta})$.
\end{lemma}

\begin{proof}
Let us for short denote $D \coloneqq (m+\sqrt{60\, m\ln \Delta})\,\frac{\Delta(G)}{r}$. Let $\bS$ be a random subset of $V(G)$ obtained by taking each vertex independently uniformly at random with probability $p=\frac{m + \sqrt{11\, m\ln \Delta}}{r}$. We note that, since $4\ln \Delta \le m\le r/3$ by assumption, we have $p<1$ and so this probability is well-defined. The result follows if there exists a realization $S$ such that every $e\in E(\cH)$ has $|e\cap S|\geq m$ and every $v\in V(G)$ satisfies $\deg_S(v)\le D$.
Let $V(G)$ be ordered arbitrarily. For every edge $e\in E(\cH)$, we let $\widetilde{e}$ consist of the first $r$ vertices (with respect to that order) of $e$.
For every edge $e\in E(\cH)$, let $\bX_e\coloneqq|\widetilde{e}\cap\bS|$ be the random variable that counts the number of vertices of $\widetilde{e}$ in $\bS$, and we write $\mu_e\coloneqq\esp{\bX_e}=m + \sqrt{11m\ln \Delta}$. For every vertex $v\in V(G)$, we let $\bX_v\coloneqq|N_G(v)\cap\bS|$ be the random variable that counts the number of neighbours of $v$ in $\bS$, and we write $\mu_v\coloneqq\esp{\bX_v}=\frac{m + \sqrt{11m\ln\Delta}}{r}\cdot \deg_G(v)\le \frac{\Delta(G)}{r}(m + \sqrt{11m\ln\Delta})$. For every edge $e\in E(\cH)$, let $B'_e$ be the random event that $\bX_e <m$, and for every vertex $v\in V(G)$, let $B_v$ be the random event that $\bX_v >D$. 

Observe that, for every pair of vertices $(u,v)$ such that $N(u) \cap N(v) = \emptyset$, the random events $B_u$ and $B_v$ are independent. So the bad event $B_v$ is dependent with at most $\Delta(G)^2$ bad events $B_u$, and at most $\Delta(G)\Delta(\cH)$ bad events $B'_e$. So the dependency-degree of $B_v$ is at most $\Delta(G)\Delta \le \Delta^2$.
Moreover, for every pair of edges $(e,f)$, such that $\widetilde{e}\cap \widetilde{f} = \emptyset$, the random events $B'_e$ and $B'_f$ are independent.
So the bad event $B'_e$ is dependent with at most $r\Delta(\cH)$ bad events $B'_f$, and at most $r\Delta(G)$ bad events $B_v$. So the dependency-degree of $B'_e$ is at most $r\Delta \le \Delta^2$.
Hence we may apply the LLL in order to show that, with non-zero probability, no event $B_v$ or $B'_e$ occurs. Regarding that the maximum degree of the dependency-graph of those random events is at most $\Delta^2$, it suffices to prove that
    \[\pr{\bX_e<m}\le\frac{1}{e\Delta^2} \quad \et \quad
    \pr{\bX_v>D}\le\frac{1}{e\Delta^2}.\]
%
%
We do so by applying Chernoff bounds on the random variable $\bX_e$ of expectancy $\mu_e$:
\begin{align*}
    \pr{\bX_e<m}
    &\le\exp\left(-\frac{(\mu_e-m)^2}{2\mu_e}\right)\\
    &=\exp\left(-\frac{11m\ln\Delta}{2m+2\sqrt{11m\ln\Delta}}\right)=\exp\left(-\frac{11\ln\Delta}{2+2\sqrt{\frac{11\ln\Delta}{m}}}\right) \\
    &\le \Delta^{-\frac{11}{2+\sqrt{11}}} \le \frac{1}{e\Delta^2} \hspace{60pt} \mbox{if $\Delta$ is large enough.}
\end{align*}
We do the same with the random variable $\bX_v$ of expectancy $\mu_v$:
\begin{align*}
    \pr{\bX_v>D}
    &=\exp\left(-\frac{(D-\mu_v)^2}{2D}\right)\\
    &\le\exp\left(-\frac{(\sqrt{60}-\sqrt{11})^2m\ln\Delta}{2m+2\sqrt{60m\ln\Delta}}\cdot \frac{\Delta(G)}{r}\right) =\exp\left(-\frac{(\sqrt{60}-\sqrt{11})^2\ln\Delta}{2+2\sqrt{\frac{60\ln\Delta}{m}}}\cdot \frac{\Delta(G)}{r}\right) \\ 
    &\le \Delta^{-\frac{\pth{\sqrt{60}-\sqrt{11}}^2}{2+2\sqrt{15}}} \le \frac{1}{e\Delta^2} \hspace{60pt} \mbox{if $\Delta$ is large enough.}
\end{align*}
By Lemma~\ref{lemma:lll}, with non-zero probability, $\bS$ satisfies the conclusion of Lemma~\ref{lem:degree-dominating-set}.
\end{proof}

We note that it is possible to drop the condition that $\Delta$ is large enough in the statement of Lemma~\ref{lem:degree-dominating-set} if we set $D\coloneqq \frac{\Delta(G)}{r}(m+\sqrt{30m(1+2\ln \Delta)})$ and $p \coloneqq \frac{1}{r}(m+\sqrt{5.5 m (1+2\ln \Delta)})$ instead, and assume that $m\ge 2+4\ln \Delta$ and $\epsilon(\cH) \ge 6+12\ln \Delta$.
\medskip

We may combine Theorem~\ref{thm:chi-bound} and Lemma~\ref{lem:degree-dominating-set} in order to obtain that, for a quasi-regular graph $G$ (that is, a graph $G$ where the ratio $\Delta(G)/\delta(G)$ is bounded), the odd chromatic number of $G$ is not too far from its chromatic number.

\begin{cor}
    \label{cor:quasi-reg}
    Let $G$ be a graph of maximum degree $\Delta$ large enough, and minimum degree $\delta \ge 12\ln (2\Delta)$. Then 
    \[\odd(G) \le \chi(G) + \ceil{4(\ln \Delta+\ln\ln \Delta + 3)} + \frac{20 \Delta \ln (2\Delta)}{\delta} = \chi(G) + O\pth{\frac{\Delta \ln \Delta}{\delta}} \quad \mbox{as $\Delta\to\infty$}.\]
\end{cor}

Note that the minimum degree condition in 
    Corollary~\ref{cor:quasi-reg} can be dropped. Indeed, if it is not fulfilled, then the upper bound on $\odd(G)$ is larger than that given by Theorem~\ref{thm:smalldegree}.

\subsection{Graphs of small chromatic number}
Given a graph-hypergraph pair $(G,\cH)$, if we can find a set $S$ that satisfies the hypothesis of Theorem~\ref{thm:chi-bound} such that $\chi(G[S])$ is much smaller than $\Delta(G[S])$, we can use another approach to obtain a better bound.

\begin{thm}
    \label{thm:chi-bound2} Let $(G,\cH)$ be a graph-hypergraph pair. Fix $\eta \coloneqq \ceil{4(\ln \Delta(\cH) + \ln \ln \Delta(\cH) + 3)}$, and assume that $\Delta(\cH) \ge 49$.
    For every subset of vertices $S\subseteq V(\cH)$, if $\epsilon(\cH[S]) \ge \eta/2$, then $(G,\cH)$ has an odd $k$-colouring, where 
    \[ k \le \chi(G\setminus S) + \eta\, \chi(G[S]).\]
\end{thm}

\begin{proof}
    Let $G_0 \coloneqq G\setminus S$ and $G_1\coloneqq G[S]$. For each $i\in \{0,1\}$, we write $k_i \coloneqq \chi(G_i)$, and we let $\sigma_i$ be a proper $k_i$-colouring of $G_i$. 

    We define a random proper $k$-colouring $\bsig$ of $G$ as follows, where $k=k_0 + \eta k_1$. For every $v\in S$, draw some random value $\bx_v$ uniformly at random from $[\eta]$, and let $\bsig(v) \coloneqq (\sigma_1(v), \bx_v)$. For every $v\notin S$, let $\bsig(v) \coloneqq \sigma_0(v)$.
    Let us order the vertices in $V(G)$ arbitrarily.
    For every $e\in E(\cH)$, we let $B_\bsig(e)$ be the random (bad) event that $e$ has no odd colour. Let $m\coloneqq \eta/2$, and let $M(e)$ contain the smallest $m$ vertices of $e\cap S$.
       Let $\sigma$ be a possible realisation of $\restrict{\bsig}{V(G)\setminus M(e)}$. By construction, for every $v\in S$, there are $\eta$ choices in $L_\bsig(v)$.
    Hence we may apply Lemma~\ref{lem:badevent} and obtain that
    \[ \pr{B_\bsig(e) ~\middle|~ \restrict{\bsig}{V(G) \setminus M(e)}=\sigma} \le \sqrt{2} \pth{\frac{2m}{e\eta}}^{m/2} = \sqrt{2}\,e^{-m/2}.\]

    As explained in the proof of Theorem~\ref{thm:main}, this is at most $\frac{1}{em\Delta(\cH)}$.
    For an edge $e'\in E(\cH)$, the outcome of $B_\bsig(e')$ is entirely determined by the realisation of $\restrict{\bsig}{e'}$. So if we fix the realisation of $\bsig$ outside of $M(e)$, we in particular fix the outcomes of all events $B_\bsig(e')$ such that $M(e) \cap e' = \emptyset$. So we set $\Gamma(e) \coloneqq \{e' : e' \cap M(e) \neq \emptyset\}$, and observe that these sets have size at most $m\Delta(\cH)$. We may now apply Lemma~\ref{lemma:lll} to the bad events $(B_\bsig(e))$, with that definition of $\Gamma(e)$, and conclude that with positive probability, no event $B_\bsig(e)$ occurs. So there is a realisation of $\bsig$ that is an odd $k$-colouring of $(G,\cH)$. This concludes the proof.
\end{proof}

\begin{cor}
    Let $G$ be a graph of maximum degree $\Delta \ge 49$, and fix $\eta \coloneqq \ceil{4(\ln \Delta + \ln \ln \Delta + 3)}$. If the minimum degree of $G$ is at least $\eta/2$, then 
    \[ \odd(G) \le \eta\, \chi(G),\]
     and this is tight up to a multiplicative constant for a family of graphs of increasing chromatic numbers.
\end{cor}

\begin{proof}
    The upper bound on $k$ is a direct consequence of Theorem~\ref{thm:chi-bound2} where $\cH$ is the neighbourhood-hypergraph of $G$, and $S=V(G)$. Let us prove the tightness of the bound.

    Fix an integer $k_0\ge 2$, and let $n_0 \ge k_0$ be an even integer.
    Let $G_0$ be a complete $k_0$-partite graph, with parts $X_1, \ldots, X_{k_0}$ all of size $2n_0$.
    For every $i\in [k_0]$, and for every $S \in \binom{X_i}{n_0}$, we add a vertex with neighbourhood $S$ in $G_0$. Let $G$ be the obtained graph; let us show that $\odd(G) \ge k_0 (n_0+1)$. 
    We write $k\coloneqq \odd(G)$, and let $\sigma$ be an odd $k$-colouring of $G$.
    First observe that we must have $\sigma(X_i) \cap \sigma(X_j) = \emptyset$ for every $i\neq j$, otherwise we would find a monochromatic edge in $\sigma$. So it suffices to show that $|\sigma(X_i)| \ge n_0+1$ for every $i\in [k_0]$.
    Let us assume for the sake of contradiction that $|\sigma(X_i)| \le n_0$. For every odd colour of $X_i$, we remove one vertex with that colour from $X_i$. We are left with at least $n_0$ vertices. We now remove monochromatic pairs of vertices from $X_i$ until exactly $n_0$ vertices remain. We obtain a set $S$ with no odd colour, and by construction there is a vertex in $V(G)$ such that $N(v)=S$. So $v$ has no odd colour in $\sigma$, a contradiction.

    The maximum degree of $G$ is $\frac{1}{2}\binom{2n_0}{n_0}+(k_0-1)\cdot 2n_0 < 4^{n_0}$ when $n_0$ is large enough, and $\chi(G)=\chi(G_0)=k_0$. So we have 
    \[ \odd(G) > \chi(G)\log_4 \Delta(G),\]
    while the minimum degree of $G$ is $n_0 \ge \log_4 \Delta(G)$.

    Note that the minimum degree of $G$ is smaller than what is required by a factor $4\ln 2$. If we want to meet the required lower bound for the minimum degree of $G$, we may let the size of each $X_i$ be $\frac{7}{6}n_0$ instead of $2n_0$. We can now prove that more than $\frac{n_0}{6}$ colours must appear on each part $X_i$.
    We still have $\chi(G)=k_0$ and $\delta(G)=n_0$, while 
    \begin{align*}
        \Delta(G) = (k_0-1)\cdot\frac{7}{6}n_0 + \binom{7n_0/6}{n_0} = (k_0-1)\cdot\frac{7}{6}n_0 + \binom{7n_0/6}{n_0/6} \le (k_0-1)\cdot\frac{7}{6}n_0 + (7e)^{n_0/6}.
    \end{align*}
    Since $6/\ln(7e) > 2.03$, when $n_0$ is large enough we have $\delta(G)=n_0 \ge 2(\ln \Delta(G) + \ln \ln \Delta(G) + 3)$, as required. On the other hand, we have $\odd(G) > \frac{1}{6} k_0 n_0$, so we are only a factor $12$ away from the upper bound guaranteed in that regime\footnote{With a more refined estimate of the binomial coefficient, one can replace $7/6$ with $53/45$, and conclude that we are only a factor $45/4$ away from best possible.}.
\end{proof}

\section{The effect of multiplying the constraints}
\label{sec:general}
\noindent We may now wonder what happens when we seek for an odd colouring of a given graph-hypergraph pair $(G,\cH)$, such that every edge $e\in E(\cH)$ has \emph{many} odd colours.
For a positive integer $h$, an $h$-\emph{odd colouring} $\sigma$ of a hypergraph $\cH$ satisfies the constraint that every edge $e\in E(\cH)$ has at least $\min\brc{h,|e|}$ odd colours in $\sigma$.
For a graph-hypergraph pair $(G,\cH)$, an \emph{$h$-odd $k$-colouring} of $(G,\cH)$ is a mapping $\sigma\colon V \to [k]$ that is both a proper colouring of $G$ and an $h$-odd colouring of $\cH$. The least $k$ for which $G$ is $h$-odd $k$-colourable is in turn called the $h$-\emph{odd chromatic number} of $(G,\cH)$ and we denote it by $\odd^h(G,\cH)$. We say that $\sigma$ is an \emph{$h$-odd $k$-colouring} of $G$ if $\sigma$ is an $h$-odd $k$-colouring of $(G,\cH)$ when $\cH$ is the neighbourhood-hypergraph of $G$. We denote $\odd^h(G)$ the $h$-\emph{odd chromatic number} of $G$. By a greedy algorithm, we can immediately get $\odd^h(G)\leq (h+1)\Delta(G)+1$ for any graph $G$ and integer $h$, or more generally that $\odd^h(G,\cH) \le h\Delta(\cH)+\Delta(G)+1$ given a graph-hypergraph pair $(G,\cH)$. We will show that we can ensure much better upper bounds with an additional reasonable minimum edge size condition in $\cH$.


We will rely on an extension of Lemma~\ref{lem:badevent} to $h$-odd colourings, that we present hereafter.
\begin{lemma}
\label{lem:hypergraphbadevent}
Let $(G,\cH)$ be a graph-hypergraph pair, and $\C(G)$ a set of colourings of $G$. Let $\bsig$ be drawn uniformly at random from $\C(G)$, and assume that there exists an integer $\tau$ such that we deterministically have $|L_\bsig(v)|\ge \tau$ for every integer $v\in V(G)$. Let $h\ge 1$ and $t\ge 0$ be integers that satisfy $m\coloneqq h-1+t \le \epsilon(\cH)$, and $B_\bsig(e)$ be the bad event that $e$ has less than $h$ odd colours in $\bsig$. Then, for every subset $M(e)\subseteq e$ of size $m$, and for every possible realisation $\sigma_0$ of $\restrict{\bsig}{V(G)\setminus M(e)}$, we have
\[ \pr{B_\bsig(e)  ~\middle|~ \restrict{\bsig}{V(G)\setminus M(e)}=\sigma_0} \le  \sqrt{2}\, \binom{m}{t}\pth{\frac{2t}{e\tau}}^{t/2},\]
for every $m \le \epsilon(\cH)$.

\end{lemma}

\begin{proof}
    Let $e\in E(\cH)$, and let $M(e)=\{u_1, \ldots, u_m\}$ be a fixed subset of $m$ vertices in $e$. Let $\sigma_0$ be a possible realisation of $\restrict{\bsig}{V(G)\setminus M(e)}$.
    Let $\bsig_0$ be drawn uniformly at random from the extensions of $\sigma_0$ to $\C(G)$.
    For every $1\le i \le m$, we let $\bsig_i \in \C(G)$ be obtained from $\bsig_{i-1}$ by resampling the colour of $u_i$ uniformly at random from $L_{\bsig_{i-1}}(u_i)$.
    For every $i\le m$, let $\bS_i$ be the number of odd colours of $e$ in $\bsig_i$.
    For every $i\ge 1$, we have $\bS_i = \bS_{i-1}-1$ if $\bsig_i(u_i)$ is one of the $\bS_{i-1}$ odd colours of $e$ in $\bsig_{i-1}$; since there are at least $\tau$ choices for $\bsig_i(u_i)$ this happens with probability at most $\frac{h-1}{\tau}$ if $\bS_{i-1}=h-1$. Otherwise, we have $\bS_i = \bS_{i-1}+1$. So the sequence $(\bS_i)_{i\le m}$ satisfies the hypotheses of Lemma~\ref{lemma:be}, which yields 
    \[ \pr{B_{\bsig_m}(e)}=\pr{\bS_m \le h-1} \le \sqrt{2}\, \binom{m}{h-1}\pth{\frac{2(m-h+1)}{e\tau}}^{\frac{m-h+1}{2}} = \sqrt{2}\,\binom{m}{t}\pth{\frac{2t}{e\tau}}^{t/2}.\]

    Since we resample the colours uniformly at random, the random colourings $(\bsig_i)_{i\le m}$ are identically distributed. Therefore, if $\bsig$ is drawn uniformly at random from $\C(G)$, we have 
    \[\pr{B_\bsig(e) ~\middle|~ \restrict{\bsig}{V(G)\setminus M(e)} = \sigma_0} = \pr{B_{\bsig_0}(e)} = \pr{B_{\bsig_m}(e)},\]
    and the conclusion follows.
\end{proof}

\begin{thm}\label{thm:rodd}
    Let $(G,\cH)$ be a graph-hypergraph pair of $\Delta(\cH)\ge 49$, and let $h$ be a given integer. 
   If there exists a subset of vertices $S\subseteq V(\cH)$ such that $\min\{h-1,\epsilon(\cH[S])- h+1\} \ge 2(\ln \Delta(\cH)+\ln \ln \Delta(\cH)+3)$, 
    then $(G,\cH)$ has an $h$-odd $k$-colouring, where 
    \[ k \le 
    \begin{cases}
    \chi(G\setminus S) + \Delta(G[S]) + 32(h-1) & \text{if $\epsilon(\cH[S])\ge 2(h-1)$;}\\

     \chi(G\setminus S) + \Delta(G[S]) + 2e^2\, \frac{\epsilon(\cH[S])^{2+1/\ln\Delta(\cH)}}{\epsilon(\cH[S])-h+1}  & \text{otherwise.}
     \end{cases}\]
\end{thm}

\begin{proof}
    Let $k_0 \coloneqq \chi(G\setminus S)$ and let $\sigma_0$ be a proper $k_0$-colouring of $G\setminus S$.
    Fix $k \coloneqq k_0 + \Delta(G[S]) + \eta$, for some integer $\eta\geq 1$ whose precise value will be determined later in the proof, and let $\bsig$ be a uniformly random proper $k$-colouring of $G$ that satisfies $\restrict{\bsig}{G\setminus S}=\sigma_0$. 
    For every edge $e\in E(\cH)$, we let $B_{\bsig}(e)$ be the (bad) random event that $e$ contains less than $h$ odd colours in $\bsig$. Let us show that, with non-zero probability, no event $B_{\bsig}(e)$ occurs.

    Let us write $t\coloneqq \min \{ \epsilon(\cH[S])-h+1, h-1 \}$ and $m\coloneqq h-1+t$. Let $e \in E(\cH)$ and $M(e) = \{v_1,v_2,\ldots, v_m\} \subseteq e\cap S$ be a subset of $m$ vertices in $e$.
    Let us recolour the vertices in $M(e)$ in turn with a uniformly random available colour. Each time we recolour $v_i$, the neighbours of $v_i$ in $S$ forbid at most $\deg_S(v_i) \le \Delta(G[S])$ colours, and the neighbours of $v_i$ not in $S$ forbid at most $k_0 = \chi(G\setminus S)$ colours (these colours are fixed by $\sigma_0$). So there are at least $\eta$ available colours for $v_i$. In particular, we have $|L_{\bsig}(v)|\geq\eta$ for each $v\in M(e)$.


    We apply Lemma~\ref{lemma:lll} with $\Gamma(e) \coloneqq  \{ e' \in E(\cH) : e' \cap M(e) \neq \emptyset\}$ for every edge $e\in E(\cH)$, and obtain that, with non-zero probability, none of the events $B_{\bsig}(e)$ occurs. The size of $\Gamma(e)$ is at most $m\Delta(\cH)$.  Let $\Sigma_0$ be the set of possible realisations of $\restrict{\bsig}{V(G) \setminus M(e)}$ such that no event $B_\bsig(e')$ occurs for $e' \notin \Gamma(e)$. For every $Z\subseteq E(\cH) \setminus \Gamma(e)$, we have
\[\pr{B_{\bsig}(e) \midbar \bigcap_{e' \in Z} \overline{B_{\bsig}(e')}} \le \sup_{\sigma_0 \in \Sigma_0}\pr{B_{\bsig}(e) \midbar \restrict{\bsig}{V(G)\setminus M(e)}=\sigma_0} \le \sqrt{2} \binom{m}{t}\pth{\frac{2t}{e\eta}}^{\frac{t}{2}},\]
 by Lemma~\ref{lem:hypergraphbadevent} applied to the graph $G$ with $\C(G)$ being the set of proper $k$-colourings of $G$. 


Let us fix $\eta \coloneqq 2t\, (\frac{m}{t}\times\binom{m}{t})^{2/t}$, so that the above probability is at most $\sqrt{2}\frac{t}{m}e^{-t/2}$. We can apply Lemma~\ref{lemma:lll} if this is at most $\frac{1}{em\Delta(\cH)}$, which is equivalent to 
 \[ \sqrt{2}e^{-t/2} \le \frac{1}{et\Delta(\cH)}.\]
As in the proof of Lemma~\ref{lemma:admissible}, this holds whenever $t\ge \ceil{-2 W_{-1}\pth{-\frac{1}{2\sqrt{2}e\Delta(\cH)}}}$. Since by hypothesis we have $\Delta(\cH) \ge 49$ and $t \ge \ceil{2(\ln \Delta(\cH) + \ln \ln \Delta(\cH) + 3)}$, this inequality is verified.


If $t=h-1$ then 
   \begin{align*}
      \eta &= 2(h-1)\pth{2\binom{2h-2}{h-1}}^{\frac{2}{h-1}} \le 2(h-1)\pth{2^{2h-2}}^{\frac{2}{h-1}} = 32(h-1).
 \end{align*}
Otherwise, we have $m = \epsilon(\cH[S]) > 2t$, and
    \[\eta = 2t\,\pth{\frac{m}{t}\binom{m}{t}}^\frac{2}{t} \le 2t\pth{\frac{m}{t}}^{\frac{2}{t}} \pth{\frac{me}{t}}^2\le2e^2\, \frac{m^{2+\frac{1}{\ln\Delta(\cH)}}}{m-h+1}.\]

This proves the existence of a proper $k$-colouring $\sigma$ of $(G,\cH)$ such that every edge has at least $h$ odd colours in $\sigma$, as desired.

\end{proof}

\rk The second bound of Theorem~\ref{thm:rodd} is at most $\chi(G\setminus S) + \Delta(G[S]) + 2e^3\, \frac{\delta(G[S])^{2}}{\delta(G[S])-h+1}$ when $\cH$ is the neighbourhood-hypergraph of $G$.

The following corollary can be derived from Theorem~\ref{thm:rodd} by setting $S\coloneqq V(G)$. 
\begin{cor}
    Let $(G,\cH)$ be a graph-hypergraph pair of $\Delta(\cH)\ge 49$, and let $h$ be a given integer. 
   Let us assume that $\min\{h-1,\epsilon(\cH)- h+1\} \ge 2(\ln \Delta(\cH)+\ln \ln \Delta(\cH)+3)$.  If $\epsilon(\cH)\ge 2(h-1)$, then there exists an $h$-odd $(\Delta(G) +32(h-1))$-colouring of $(G,\cH)$; Otherwise, $(G,\cH)$ admits an $h$-odd $(\Delta(G) + 2e^2\, \frac{\epsilon(\cH)^{2+1/\ln\Delta(\cH)}}{\epsilon(\cH)-h+1})$-colouring.
 
\end{cor}

We also extend Theorem~\ref{thm:chi-bound2} to $h$-odd colourings.

\begin{thm}
    \label{thm:chi-bound2-h} Let $(G,\cH)$ be a graph-hypergraph pair of $\Delta(\cH)\ge 49$, and let $h$ be a given integer.
    If  there exists a subset of vertices $S\subseteq V(\cH)$ such that $\min\{h-1,\epsilon(\cH[S])- h+1\} \ge 2(\ln \Delta(\cH)+\ln \ln \Delta(\cH)+3)$,  then $(G,\cH)$ has an $h$-odd $k$-colouring, 
    where 
    \[ k \le 
    \begin{cases}
    \chi(G\setminus S) + 32(h-1)\, \chi(G[S]) & \text{if $\epsilon(\cH[S])\ge 2(h-1)$;}\\

     \chi(G\setminus S) + 2e^2\, \frac{\epsilon(\cH[S])^{2+1/\ln\Delta(\cH)}}{\epsilon(\cH[S])-h+1}\, \chi(G[S])  & \text{otherwise.}
     \end{cases}\]
\end{thm}

\begin{proof}
    Let $G_0 \coloneqq G\setminus S$ and $G_1\coloneqq G[S]$. For each $i\in \{0,1\}$, we write $k_i \coloneqq \chi(G_i)$, and we let $\sigma_i$ be a proper $k_i$-colouring of $G_i$. 

    We define a random proper $k$-colouring $\bsig$ of $G$ as follows, where $k=k_0 + \eta k_1$ and $\eta$ is some positive integer whose precise value will be determined later in the proof. For every $v\in S$, draw some random value $\bx_v$ uniformly at random from $[\eta]$, and let $\bsig(v) \coloneqq (\sigma_1(v), \bx_v)$. For every $v\notin S$, let $\bsig(v) \coloneqq \sigma_0(v)$.
    Let us order the vertices in $V(G)$ arbitrarily.
    For every $e\in E(\cH)$, we let $B_\bsig(e)$ be the random (bad) event that $e$ contains less than $h$ odd colours in $\bsig$. Let us write $t\coloneqq \min \{ \epsilon(\cH[S])-h+1, h-1 \}$, and let us fix $m\coloneqq h-1+t$. We let $M(e)$ contain the smallest $m$ vertices of $e\cap S$. Let $\sigma$ be a possible realisation of $\restrict{\bsig}{V(G)\setminus M(e)}$. By construction, for every $v\in S$, there are $\eta$ choices in $L_\bsig(v)$. 

    For an edge $e'\in E(\cH)$, the outcome of $B_\bsig(e')$ is entirely determined by the realisation of $\restrict{\bsig}{e'}$. So if we fix the realisation of $\bsig$ outside of $M(e)$, we in particular fix the outcomes of all events $B_\bsig(e')$ such that $M(e) \cap e' = \emptyset$. So we set $\Gamma(e) \coloneqq \{e' : e' \cap M(e) \neq \emptyset\}$, and observe that these sets have size at most $m\Delta(\cH)$. Let $\Sigma_0$ be the set of possible realisations of $\restrict{\bsig}{V(G) \setminus M(e)}$ such that no event $B_\bsig(e')$ occurs for $e' \notin \Gamma(e)$. Hence we may apply Lemma~\ref{lem:hypergraphbadevent} and obtain that for every $Z\subseteq E(\cH) \setminus \Gamma(e)$, we have 
    \[\pr{B_{\bsig}(e) \midbar \bigcap_{e' \in Z} \overline{B_{\bsig}(e')}} \le \sup_{\sigma_0 \in \Sigma_0}\pr{B_{\bsig}(e) \midbar \restrict{\bsig}{V(G)\setminus M(e)}=\sigma_0} \le \sqrt{2} \binom{m}{t}\pth{\frac{2t}{e\eta}}^{\frac{t}{2}}.\]

     As explained in the proof of Theorem~\ref{thm:rodd}, this is at most $\frac{1}{em\Delta(\cH)}$ by fixing $\eta \coloneqq 2t\, (\frac{m}{t}\times\binom{m}{t})^{2/t}$ when $\Delta(\cH)\ge49$. We then have $\eta\le 32(h-1)$ if $\epsilon(\cH[S])\ge 2(h-1)$; otherwise $\eta \le 2e^2\, \frac{\epsilon(\cH[S])^{2+1/\ln\Delta(\cH)}}{\epsilon(\cH[S])-h+1}$. We may now apply Lemma~\ref{lemma:lll} to the bad events $(B_\bsig(e))$, with that definition of $\Gamma(e)$, and conclude that with positive probability, no event $B_\bsig(e)$ occurs. So there is a realisation of $\bsig$ that is an $h$-odd $k$-colouring of $(G,\cH)$. This concludes the proof.
\end{proof}

\subsection{Constructions}
In the current section, we have derived upper bound for $\odd^h(G,\cH)$ given an integer $h$ and a graph-hypergraph pair $(G,\cH)$ that satisfies that $\epsilon(\cH)-h$ is sufficiently large. We now show that, if no such restriction holds, then the bound obtained by a greedy colouring may be close to best possible.

We propose a construction that relies on the existence of Steiner $2$-designs, which was proven in \cite{Wil75}. Given an integer $q$, a Steiner $2$-design on $[q]$ is a collection of sets of uniform size $k$ (that we call \emph{blocks}), such that every pair of vertices from $[q]$ is contained in exactly one set. We denote it $S(2,k;q)$. Among its many properties, it must contain exactly $\binom{q}{2}/\binom{k}{2}$ sets, and each vertex is contained in exactly $\frac{q-1}{k-1}$ sets.

\begin{prop}
    For every integer $h\ge 1$, there exists a family of graphs $G$ of increasing maximum degree $\Delta$ and of minimum degree $h+1$, such that 
    \[ \odd^h(G) \ge h\Delta+1. \]
\end{prop}

\begin{proof}
    By \cite[Theorem~2.1]{Wil75}, we may find an arbitrarily large integer $q$ such that a Steiner $2$-design $S(2,h+1;q)$ exists. We let $H=(X,Y,E)$ be its (bipartite) incidence graph. We have $|X|=q$, $|Y|=\binom{q}{2}/\binom{h+1}{2}$, every vertex $y\in Y$ has degree $h+1$, every vertex $x\in X$ has degree $\frac{q-1}{h}$, and every pair of vertices $x,x'\in X$ is contained in the neighbourhood of some vertex $y\in Y$. In particular, $\delta(H)=h+1$ and $\Delta(H)=\frac{q-1}{h}$.
    
    Let $\sigma$ be an $h$-odd $k$-colouring of $H$, where $k=\odd^h(H)$. So every neighbourhood of size $h+1$ must be rainbow in $\sigma$ (i.e. contains no pair of vertices with the same colour), otherwise the number of odd colours in this neighbourhood is at most $h-1$. Since every pair of vertices from $X$ is contained in a neighbourhood of size $h+1$, we infer that $X$ is rainbow, and so $k\ge |X|=q=h\Delta(H)+1$.

\end{proof}

One interesting special case of Theorem~\ref{thm:rodd} is the following.

\begin{cor}
    \label{cor:h-odd-delta}
    Let $G$ be a $\Delta$-regular graph, let $\lceil 2(\ln \Delta +\ln \ln \Delta + 3\rceil \le t\le \Delta$ be a given integer, and let $h\coloneqq \Delta+1-t$. Then
    \[\odd^h(G) = O\pth{\frac{\Delta^2}{t}}.\]
\end{cor}

We now show that Corollary~\ref{cor:h-odd-delta} is tight up to a multiplicative constant.

\begin{prop}
For every even integer $\Delta\ge 2$ and $1\le t \le \Delta$, there is a $\Delta$-regular graph $G$ such that, letting $h\coloneqq \Delta+1-t$, one has
\[ \odd^h(G) > \frac{1}{2}\,\frac{\Delta^2}{t+1}. \]
\end{prop}
\begin{proof}
    Let $n\coloneqq \frac{\Delta}{2}+1$, and let $G\coloneqq L(K_{n,n})$ be the line-graph of a complete bipartite graph. Let $k\coloneqq \odd^h(G)$, and let $\sigma$ be a $h$-odd $k$-colouring of $G$.

    Let $v\in V(G)$. The neighbourhood of $v$ can be covered with two cliques, so each colour in $N(v)$ has at most $2$ occurrences. Let $s_\sigma(v)$ denote the number of colours with $2$ occurrences in $N(v)$; we must have $s_\sigma(v) \le (t-1)/2$. So
    \begin{equation}
        \label{eq:monopairs}
        S \coloneqq \sum_{v\in V(G)} s_\sigma(v) \le \frac{t-1}{2}n^2.
    \end{equation}
    Let $(M_1, \ldots, M_k)$ be the colour classes of $\sigma$. If $M_i$ has size $m_i$, then it has a contribution of $m_i(m_i-1)$ to $S$. Indeed, $M_i$ is a matching of size $m_i$ in $K_{n,n}$, and there are $m_i(m_i-1)$ edges incident to $2$ edges from $M_i$ in $K_{n,n}$; each of them corresponds to a vertex in $G$ with a monochromatic pair of colour $i$ in its neighbourhood. So by convexity we have 
    \[ S = \sum_{i=1}^k m_i(m_i-1) \ge k\cdot \frac{n^2}{k} \pth{\frac{n^2}{k}-1} = n^2\pth{\frac{n^2}{k}-1}. \]
    Combining this with \eqref{eq:monopairs}, we obtain
    \[ k \ge \frac{2n^2}{t+1} > \frac{1}{2}\frac{\Delta^2}{t+1}. \]
\end{proof}

We finish this section with the following consequence of Theorem~\ref{thm:chi-bound2-h}.

\begin{cor}\label{cons:hodd-multiplicative}
    Let $G$ be a graph of maximum degree $\Delta$, and let $h\ge 2(\ln \Delta(\cH)+\ln \ln \Delta(\cH)+3)$ be a given integer. If the minimum degree of $G$ is at least $2h$, then 
    \[ \odd^h(G) \le 32(h-1)\, \chi(G),\]
     and this is tight up to a multiplicative constant for a family of graphs of increasing chromatic numbers.
\end{cor}

\begin{figure}[ht]
\begin{center}
\noindent
\includegraphics[width=280.6pt]{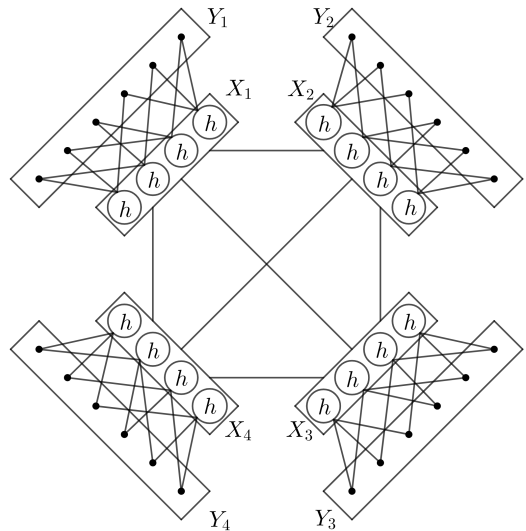}
\noindent
\setlength{\unitlength}{1pt}
\captionsetup{font={small}}
\caption{The construction in Corollary~\ref{cons:hodd-multiplicative} when $k_0=n_0=4$. A line between two sets stands for a complete bipartite graph.}\label{pic:cons-cor-6.8}
\end{center}
\end{figure}

\begin{proof}
    The upper bound can be derived from Theorem~\ref{thm:chi-bound2-h} by letting $\cH$ be the neighbourhood-hypergraph of $G$, and setting $S\coloneqq V(G)$.
    
    To prove the tightness of the bound, we show that given any $\eps>0$, there is a graph $G$ such that 
    \[\odd^h(G) \ge (2-\eps)h\,\chi(G).\]
    Fix an integer $k_0\ge 2$, and let $n_0 \ge k_0$ be an even integer.
    Let $G_0$ be a complete $k_0$-partite graph, with parts $X_1, \ldots, X_{k_0}$.
    For every $i\in [k_0]$, $X_i$ contains $n_0$ $h$-sets and the size of $X_i$ is $hn_0$. For every $i\in [k_0]$ and for every pair of $h$-sets $S$ in $X_i$, we add a vertex with neighbourhood $S$ in $G_0$. Note that for each $X_i$, we add $\binom{n_0}{2}$ vertices and put these vertices in a vertex set $Y_i$. Let $G$ be the obtained graph (see Fig.~\ref{pic:cons-cor-6.8}); let us show that 
    
    \[\odd^h(G) \ge \pth{2-\frac{2}{n_0}}h k_0.\]

    We write $k\coloneqq \odd^h(G)$, and let $\sigma$ be an $h$-odd $k$-colouring of $G$.
    First observe that we must have $\sigma(X_i) \cap \sigma(X_j) = \emptyset$ for every $i\neq j$, otherwise we would find a monochromatic edge in $\sigma$. 
    So it suffices to show that $|\sigma(X_i)| \ge \pth{2-\frac{2}{n_0}}h$ for every $i\in [k_0]$. On the one hand, each vertex $v\in Y_i$ has at least $h$ odd colours. On the other hand, for each $X_i$ and each colour $c$ in $\sigma$, let $0\le p\le n_0$ be the number of $h$-sets in $X_i$ where $c$ is an odd colour. Clearly, $c$ is an odd colour of $p(n_0-p)\le \frac{n_0^2}{4}$ vertices in $Y_i$. So we need at least $\frac{4\binom{n_0}{2}}{n_0^2} h = \pth{2-\frac{2}{n_0}}h$ colours to make every vertex have $h$ odd colours in $\sigma(X_i)$.
    We also have $\chi(G)=\chi(G_0)=k_0$, so this concludes the proof. 
    

\end{proof}



    

\bibliographystyle{abbrv}
\bibliography{odd}
\end{document}